\documentclass[a4paper]{amsart}
\usepackage{amsmath,amsfonts,amssymb,amsthm,amscd}
\usepackage{mathtools}
\usepackage{dsfont,mathrsfs,enumerate}
\usepackage{curves,epic,eepic}
\usepackage{xspace}
\usepackage{graphicx}
\usepackage{csquotes}
\usepackage{hyperref}

\newtheorem*{thmnonr}{Theorem}
\newtheorem*{mainthm}{Main Theorem}
\newtheorem{thm}{Theorem}[section]

\newtheorem{prop}[thm]{Proposition}
\newtheorem{cor}[thm]{Corollary}
\theoremstyle{definition}
\newtheorem{defi}[thm]{Definition}
\newtheorem{rem}[thm]{Remark}

\newtheorem{exa}[thm]{Example}

\newtheorem*{exanonr}{Example}

\def\SL{\mathrm{SL}}
\def\N{\mathds{N}}
\def\F{\mathds{F}}

\def\Z{\mathds{Z}}

\def\R{\mathds{R}}
\def\C{\mathds{C}}
\def\phi{\varphi}

\def\P{\mathds{P}}

\def\fractional#1{\left\{ #1 \right\}}

\DeclareMathOperator*{\diag}{diag}
\DeclareMathOperator*{\Hom}{Hom}
\DeclareMathOperator*{\Tr}{Tr}

\DeclareMathOperator*{\conv}{conv}
\DeclareMathOperator*{\supp}{supp}

\DeclareMathOperator*{\vol}{vol}
\DeclareMathOperator*{\age}{\alpha}

\def\floor#1{\left\lfloor #1 \right\rfloor}%
\def\fractional#1{\left\{ #1 \right\}}%

\def\ie{i.\,e.\xspace}
\def\eg{e.\,g.\xspace}

\begin{document}

\title[Lattice polytopes, finite abelian subgroups in $\SL(n,\C)$, coding theory]
{Lattice polytopes, finite abelian subgroups in $\SL(n,\C)$ and coding theory}

\author[Victor Batyrev]{Victor Batyrev}
\address{Mathematisches Institut, Universit\"at T\"ubingen,
Auf der Morgenstelle 10, 72076 T\"ubingen, Germany}
\email{batyrev@everest.mathematik.uni-tuebingen.de}

\author[Johannes Hofscheier]{Johannes Hofscheier}
\address{Mathematisches Institut, Universit\"at T\"ubingen,
Auf der Morgenstelle 10, 72076 T\"ubingen, Germany}
\email{johannes.hofscheier@uni-tuebingen.de}

\subjclass[2000]{Primary 52B20; Secondary 14B05, 11B68}

\begin{abstract}
  We consider $d$-dimensional lattice polytopes $\Delta$ with
  $h^*$-polynomial $h^*_\Delta=1+h_k^*t^k$ for $1<k<(d+1)/2$ and
  relate them to some abelian subgroups of $\SL_{d+1}(\C)$ of order
  $1+h_k^*=p^r$ where $p$ is a prime number. These subgroups can be
  investigate by means of coding theory as special linear constant
  weight codes in $\F_p^{d+1}$. If $p =2$, then the classication of
  these codes and corresponding lattice polytopes can be obtained
  using a theorem of Bonisoli. If $p > 2$, the main technical tool in
  the classification of these linear codes is the non-vanishing
  theorem for generalized Bernoulli numbers $B_{1,\chi}^{(r)}$
  associated with odd characters $\chi:\F_q^*\to\C^*$ where $q=p^r$.
  Our result implies a complete classification of all lattice
  polytopes whose $h^*$-polynomial is a binomial.
\end{abstract}

\date{\today}

\maketitle

\thispagestyle{empty}

\section*{Introduction}
\label{sec:intro}

Let $M\cong\Z^d$ be a $d$-dimensional lattice and $\Delta\subset
M_\R:=M\otimes\R \cong \R^d$ a $d$-dimensional lattice polytope, \ie
the vertices of $\Delta$ are contained in the lattice $M$. It is
well-known (see \eg \cite[Section
3]{BeckRobins:ComputingTheContinuousDiscretely}) that the
\emph{Ehrhart series}
\begin{align*}
  \mathrm{Ehr}_\Delta(t)=1+\sum_{k\ge1}|k\Delta\cap M|t^k
\end{align*}
is a rational function of the form
\begin{align*}
  \mathrm{Ehr}_\Delta(t)= \frac{1+h_1^*t+\ldots+h_d^*t^d}{(1-t)^{d+1}}
  = \frac{h_\Delta^*(t)}{(1-t)^{d+1}}
\end{align*}
where the coefficients $h_1^*,\ldots,h_d^*$ of the polynomial
$h_\Delta^*(t)$ are nonnegative integers. We call the polynomial
$h_\Delta^*(t)$ the \emph{$h^*$-polynomial} of $\Delta$.  Recall some
basic facts about $h^*$-polynomials.

Two lattice polytopes $\Delta\subset M_\R$ and $\Delta'\subset M'_\R$
are called isomorphic if there exists a bijective affine linear map
$\varphi:M_\R\to M'_\R$ such that $\varphi(\Delta)=\Delta'$ and
$\varphi(M)=M'$.  If $\Delta$ and $\Delta'$ are isomorphic lattice
polytopes, then $|k\Delta\cap M|=|k\Delta'\cap M'|$ for all $k\ge1$
and we have $h_\Delta^*(t)=h_{\Delta'}^*(t)$.

For any $d$-dimensional lattice polytope $\Delta\subseteq M_\R$, we
construct a new $(d+1)$-dimensional lattice polytope
\begin{align*}
  \Delta'\coloneqq\conv(\Delta\times\{0\},(\mathbf{0},1))\subseteq
  M_\R\oplus\R
\end{align*}
which is called the \emph{pyramid} over $\Delta$. It is easy to show
that $(1-t) \mathrm{Ehr}_{\Delta'}(t)=\mathrm{Ehr}_\Delta(t)$ (see \eg
\cite[Theorem 2.4]{BeckRobins:ComputingTheContinuousDiscretely}).  So
we obtain again $h_\Delta^*(t)=h_{\Delta'}^*(t)$.

Let $\mathbf{e}_1, \ldots,\mathbf{e}_d$ be a $\Z$-basis of $M$.  A
$d$-dimensional simplex $\Delta$ is called a \emph{unimodular simplex}
if it is isomorphic to the convex hull of $\mathbf{0},\mathbf{e}_1,
\ldots, \mathbf{e}_d$.  For any lattice polytope $\Delta$ we denote by
$\mathrm{vol}(\Delta)$ the \emph{integral volume} of $\Delta$, \ie
$\mathrm{vol}(\cdot)$ is the $d!$-multipliple of the standard Lebesgue
volume on $\R^d$ such that for any unimodular simplex $\Delta$ holds
$\mathrm{vol}(\Delta) =1$.  By \cite[Corollary
3.21]{BeckRobins:ComputingTheContinuousDiscretely}, we have
$h^*_\Delta(1)=1+\sum_{k=1}^dh_k^*=\vol(\Delta)$.

The knowledge of the $h^*$-polynomial of $\Delta$ imposes strong
conditions on the polytope $\Delta$. For example, the equalities
$h_1^*=\ldots=h_d^*=0$ hold if and only if $\Delta$ is a unimodular
simplex.

The purpose of this paper is to classify (up to isomorphism) lattice
polytopes $\Delta$ whose $h^*$-polynomial has exactly one nonzero
coefficient $h_k^*\neq0$ for $k \in \{ 1,\ldots,d \}$, \ie
$h_\Delta=1+h_k^*t^k$. Such an $h^*$-polynomial will be called an
\emph{$h^*$-binomial}.

Simplest examples of $h^*$-binomials are the ones of degree $k =1$,
i.e. linear $h^*$-polynomials.  Lattice polytopes $\Delta$ such that
$h^*_\Delta$ is linear have been classified by the first author and
Benjamin Nill in \cite{BatyrevNill:LatticePolytopes} (we will give a
summary of these results in section \ref{sec:recoll}). Thus, for our
purpose, it remains to investigate $d$-dimensional lattice polytopes
having an $h^*$-binomial of degree $k > 1$. One can show that the
condition $k>1$ implies that $\Delta$ is a simplex and the degree $k$
of $h^*_\Delta$ is at most $(d+1)/2$. In particular, $3$-dimensional
lattice simplices $\Delta \subset \R^3$ having nonlinear
$h^*$-binomial are precisly {\em empty lattice tetrahedra} with
$h^*$-polynomial $h_\Delta^*=1+h_2^*t^2$.  These tetrahedra have been
classified by G.\,K.~White in \cite{White:LatticeTetrahedra}.  In
\cite{BH:GeneralizedWhite}, we generalized the classification of White
to arbitrary odd dimension $d \geq 3$ using the so called ``Terminal
Lemma'' of M. Reid \cite{Reid:YoungPersonsGuide}. This implies the
classification of $d$-dimensional lattice simplices having an
$h^*$-binomial of degree $k=(d+1)/2 >1$ in (see also section
\ref{sec:recoll}).  Therefore, this paper is devoted mainly to lattice
simplices having an $h^*$-binomial with degree $1<k<(d+1)/2$ which are
not pyramids over lower-dimensional lattice simplices.

Let us give an overview of our approach. For a $d$-dimensional lattice
simplex $\Delta\subseteq M_\R$ with vertices
$\mathbf{v}_0,\ldots,\mathbf{v}_d$, we define a subgroup
$G_\Delta\subseteq\SL_{d+1}(\C)$ which consists of all diagonal
matrices
\begin{align*}
  g (\lambda):= \diag(e^{2\pi i\lambda_0},\ldots, e^{2\pi i
    \lambda_d})
\end{align*}
such that $\lambda_i\in[0,1[$ for all $i \in \{ 0,\ldots,d\}$,
$\sum_{i=0}^d\lambda_i\mathbf{v}_i\in M$ and
$\sum_{i=0}^d\lambda_i\in\Z$. The simplex $\Delta$ is determined up to
isomorphism by the subgroup
$G_\Delta\subseteq\SL_{d+1}(\C)$. Moreover, the coefficients of its
$h^*$-polynomial $h_\Delta^*=1+\sum_{k=1}^dh_k^*t^k$ can be computed
as follows
\begin{align*}
  h_k^*=|\{\diag(e^{2\pi i\lambda_0},\ldots,e^{2\pi i\lambda_d})\in
  G_\Delta:\sum_{i=0}^d\lambda_i=k\}|, \;\; 0 \leq k \leq d.
\end{align*}
It is easy to show that a lattice simplex $\Delta$ is a pyramid over a
lower-dimensional simplex if and only if there exists $i \in \{
1,\ldots,d+1 \}$ such that $\lambda_i=0$ for all $g(\lambda) \in
G_\Delta$.

For a prime number $p$, we denote the group of $p$-th roots of unity
in $\C$ by $\mu_p$.
\begin{thmnonr}
  Let $\Delta$ be a $d$-dimensional lattice simplex which is not a
  pyramid over a lower-dimensional simplex. If the degree $k$ of the
  $h^*$-polynomial $h_\Delta^*=1+h_k^*t^k$ satisfies $1<k<(d+1)/2$,
  then there exists a prime number $p$ such that all nontrivial
  elements $g(\lambda)\in G_\Delta$ have order $p$.  In particular,
  $G_\Delta$ can be considered as a subgroup of $\mu_p^{d+1}$ and one
  obtains $\vol(\Delta)=1+h_k^*=| G_\Delta | = p^r$ for a positive
  integer $r$.
\end{thmnonr}

For a prime number $p$, we denote by $\F_p$ the finite field of order
$p$.  By the above theorem, the subgroup $G_\Delta \subseteq
\mu_p^{d+1}$ can be identified with an $r$-dimensional linear subspace
$L_\Delta \subseteq \F_p^{d+1}$.

Let $L\subseteq\F_p^n$ be an arbitrary $r$-dimensional linear subspace
in $\F_p^n$. Choose a basis $\mathbf{a}_1,\ldots,\mathbf{a}_r$ of
$L$. For $i=1,\ldots,r$ we consider the vector $\mathbf{a}_i =
(a_{i1}, \ldots, a_{in})$ as the $i$-th row of a $(r\times n)$-matrix
$A$ which we call a \emph{generator matrix} of $L$.  The linear
subspace $L$ is uniquely determined by the matrix $A$.  Let us write
an arbitary vector $ \mathbf{v}$ in $\F_p^n$ as
\begin{align*}
  \mathbf{v} = (\overline{v_1}, \ldots, \overline{v_n}) \in
  (\Z/p\Z)^n, \;\; 0 \leq v_i < p \;\; \forall 1 \leq i \leq n.
\end{align*} 
We consider the following two functions on $\F^n_p$ with values in
$\Z_{\geq 0}$:
\begin{itemize}
\item the \emph{weight}-function
  \begin{align*}
    \omega(\mathbf{v}) := |\{ i \in \{ 1, \ldots, n \} \; : \; v_i
    \neq 0 \}|
  \end{align*}
\item and the \emph{age}-function
  \begin{align*}
    \alpha(\mathbf{v}) := \sum_{i=0}^dv_i.
  \end{align*}
\end{itemize}
We say that a linear subspace $L \subset \F_p^n$ has \emph{constant
  weight} (resp.  \emph{constant age}), if the \emph{weight}-function
$\omega$ (resp. \emph{age}-function $\alpha$) is constant on the set
of all nonzero vectors in $L$.  It is easy to show that if a linear
subspace $L\subseteq \F_p^n$ has constant age then $L$ has also
constant weight and one has
\begin{align*}
  p \omega (\mathbf{v})=2\alpha(\mathbf{v}) \;\; \forall \mathbf{v}
  \in L \setminus \{0\}.
\end{align*}
We will see below some examples of linear subspaces $L$ in $\F^n_p$ of
constant weight that do not have constant age.

In coding theory linear subspaces of $\F_p^n$ are called \emph{linear
  codes}. It is important to remark that the linear code $L_\Delta
\subset F_p^{d+1}$ associated to a $d$-dimensional lattice simplex
$\Delta$ with $h^*$-binomial of degree $1<k<(d+1)/2$ has constant age
$kp$.  There are special linear codes of constant weight which are
called \emph{simplex codes}. These linear codes are constructed as
follows.  Fix a positive integer $l \in\N$ and put $m:= (p^l-1)/(p-1)$
to be the number of points in $(l-1)$-dimensional projective space
$\P^{l-1}$ over $\F_p$. For any point $x \in \P^{l-1}$ one chooses a
nonzero vector $A(x) \in \F^l_p$ in the corresponding $1$-dimensional
linear subspace in $\F_p^l$. We consider the vectors $A(x_1), \ldots,
A(x_m)$ as columns of a $l \times m$-matrix $A$ and define the
$l$-dimensional linear simplex code in $\F_p^m$ by the generator
matrix $A$. We remark that the above matrix $A$ is determined uniquely
up to permutations of its columns and multiplications by nonzero
elements in $\F_p$.

If $p=2$, then the \emph{weight}-function $\omega$ and the
\emph{age}-function $\alpha$ coincide.  We observe that in this case
the generator matrix of a simplex code is unique up to permutation of
the columns and the classification of linear codes of constant age can
be obtained using a theorem of Bonisoli (see \eg
\cite{Bonisoli:EquidistantLinearCodes} or
\cite{WardWood:EquivalenceOfCodes}).

\begin{mainthm}[Case $p=2$]
  Let $\Delta$ be a $d$-dimensional lattice simplex with
  $h^*$-binomial of degree $1<k<(d+1)/2$. Assume that
  $\mathrm{vol}(\Delta)=2^r$ for a positive integer $r$ and $\Delta$
  is not a pyramid over a lower-dimensional simplex. Then the numbers
  $k, d, r$ are related by the equation:
  \begin{align*}
    2^{r-2}(d+1)=k(2^r-1)
  \end{align*}
  and the generating $(r \times (d+1))$-matrix of the linear code
  $L_\Delta$ can be written up to a permutation of the columns in the
  form $(S_r,\ldots ,S_r)$ where $S_r$ is repeated
  $(k/2^{r-2})$-times.
\end{mainthm}

\begin{exanonr}
  Let $\Delta$ be a $d$-dimensional lattice simplex with
  $h^*$-binomial of degree $1<k<(d+1)/2$. Assume that
  $\mathrm{vol}(\Delta)=2$ (i.e. $r =1$, or $h_\Delta^*=1+t^k$) and
  $\Delta$ is not a pyramid over a lower-dimensional simplex. Then the
  above theorem implies that $d+1=2k$ and $G_\Delta=\{\pm E \}$.
\end{exanonr}

The case $p>2$ is more involved.

\begin{exanonr} If $p =3$ then the matrix
  \begin{align*}
    A_1 =
    \begin{pmatrix}
      1 & 0 & 1 & 2 \\
      0 & 1 & 1 & 1
    \end{pmatrix}
  \end{align*}
  generates a $2$-dimensional simplex code $L_1 \subset \F_3^4$ of
  constant weight $3$. But the rows of $A_1$ have different age.  So
  the linear code $L_1$ does not have constant age.
\end{exanonr}

\begin{exanonr}
  We can modify the previous example as follows. Consider the
  $2$-dimensional linear code $L$ in $\F_3^8$ generated by the matrix:
  \begin{align*}
    A = (A_1, -A_1) =
    \begin{pmatrix}
      1 & 0 & 1 & 2 &\; &  2 & 0 & 2 & 1 \\
      0 & 1 & 1 & 1 &\; & 0 & 2 & 2 & 2
    \end{pmatrix}
  \end{align*}
  Then the linear code $L$ has not only the constant weight $6$, but
  it has the constant age $9$, because for any vector
  $(\overline{v_1}, \ldots, \overline{v_8}) \in L$ one has $v_i + v_{i
    +4} \in \{ 0, 3 \}$ for all $1 \leq i \leq 4$.
\end{exanonr}

Our main theorem in the case $p>2$ claims that the last example is in
some sense typical for constucting any linear code of constant age:

\begin{mainthm}
  Let $\Delta$ be a $d$-dimensional lattice simplex having an
  $h^*$-binomial of degree $1<k<(d+1)/2$. Assume that
  $\mathrm{vol}(\Delta)=p^r$ for a prime number $p$ and a positive
  integer $r$. Assume that $\Delta$ is not a pyramid over a
  lower-dimensional simplex. Then the number $p,d,k, r$ are related by
  the equation
  \begin{align*}
    (p^r-p^{r-1})(d+1)=2k(p^r-1)
  \end{align*}
  and the generating $(r \times (d+1))$-matrix of the linear code
  $L_\Delta$ can be written up to permutation of the columns in the
  form
  \begin{align*}
    (A_1, -A_1, A_2, -A_2, \ldots, A_s, -A_s ),
  \end{align*}
  where $s=k/p^{r-1}$ and $A_1, \ldots, A_s$ are generator matrices of
  $r$-dimensional simplex codes.
\end{mainthm}

In the proof of this theorem we will use some number theoretic results
which have probably independent interest. Fix an odd prime $p$ and a
prime power $q=p^r$. Let $B_1$ be the \emph{$1$st (periodic) Bernoulli
  function} which maps a real number $x$ to
\begin{align*}
  B_1(x)=
  \begin{cases}
    \fractional{x}-\frac{1}{2} & x\not\in\Z\\
    0 & x\in\Z
  \end{cases}
\end{align*}
where $\fractional{x}$ denotes the \emph{fractional part} of $x$, \ie
$\fractional{x}=x-\floor{x}$ where $\floor{x}$ is the biggest integer
which is smaller than or equal to $x$. For any $a \in \F_q$ we denote
by ${\rm Tr}(a) \in \{ 0, 1, \ldots, p-1 \}$ the representative of the
trace of $a$ in $\F_p$. We define generalized Bernoulli numbers
associated to characters $\chi:\F_q^*\to\C^*$ as
\begin{align*}
  B_{1, \chi}^{(r)} \coloneqq \sum_{a \in \F_q^*} \chi(a) B_1 \left(
    \frac{\Tr(a)}{p} \right).
\end{align*}
The number $B_{1, \chi}^{(1)}$ coincides with the classical
generalized Bernoulli number $B_{1, \chi}$ (see \eg \cite[Chapter
4]{Washington:CyclotomicFields}).  The important technical tool in the
proof of the main theorem is the following non-vanishing theorem
similar to the classical case (see \eg
\cite{Washington:CyclotomicFields}).
\begin{thmnonr}
  Let $\chi:\F_q^*\to\C^*$ be an arbitrary odd character, \ie
  $\chi(-1)=-1$. Then $B_{1,\chi}^{(r)}\neq 0$.
\end{thmnonr}

This paper is organized as follows.  In Section \ref{sec:recoll}, we
recall some already known classification results of lattice polytopes
having an $h^*$-binomial of degree $k$. In Section \ref{sec:G_Delta},
we define the group $G_\Delta$ associated to a lattice simplex
$\Delta$ and make some basic observations. In Section
\ref{sec:k_in_between}, we investigate the group $G_\Delta$ of a
lattice simplex $\Delta$ having an $h^*$-binomial of degree
$1<k<(d+1)/2$ and show how these groups a related to coding theory. In
Section \ref{sec:coding_theory}, we recall the notions and results
from coding theory which we will need later on. In Section
\ref{sec:proof_of_main_thm}, we classify linear codes of constant age
and prove our main classification result. Finally, we will prove the
non-vanishing theorem for the generalized Bernoulli numbers
$B_{1,\chi}^{(r)}$ in section \ref{sec:non_vanishing_of_B_1_chi}.

\section{\texorpdfstring{The cases $k =1$ and $k =(d+1)/2$}{The cases k=1 and k=(d+1)/2}}
\label{sec:recoll}

Let $\Delta$ be a $d$-dimensional lattice polytope with $h^*$-binomial
of degree $k$.  The case $k\le1$ has been studied by the first author
and Benjamin Nill in \cite{BatyrevNill:LatticePolytopes}.  We shortly
recall their results.

\begin{defi}
  Let $M$ be a lattice.  Consider $r+1$ lattice polytopes
  $\Delta_0,\ldots,\Delta_r\subset M_\R$ and the cone
  \begin{align*}
    \sigma \coloneqq \{ (\lambda_0, \ldots, \lambda_r,
    \sum\lambda_i\Delta_i) \subset \R^{r+1} \oplus M_\R : \lambda_i\ge
    0\}.
  \end{align*}
  The polytope which arises by intersecting the cone $\sigma$ with the
  hyperplane $H\coloneqq\{\mathbf{x}\in\R^{r+1}\oplus
  M'_\R:\sum_ix_i=1\}$ is called the \emph{Cayley polytope} of
  $\Delta_0,\ldots,\Delta_r$ and we denote it by
  $\Delta_0*\ldots*\Delta_r$.
  
  If $\dim\Delta_i=0$ for $i=1,\ldots,r$, we call
  $\Delta_0*\ldots*\Delta_r$ the \emph{$r$-fold pyramid over
    $\Delta_0$}.
\end{defi}

\begin{defi}
  An $n$-dimensional lattice polytope $\Delta$ ($n\ge1$) is called
  \emph{Lawrence prism} with \emph{heights} $h_1,\ldots,h_n$, if there
  exist non-negative integers $h_1,\ldots,h_n$ such that $\Delta$ is
  the Cayley polytope of $n$ segments
  \begin{align*}
    [0,h_1],\ldots,[0,h_n]\subset\R.
  \end{align*}
\end{defi}
  
\begin{defi}
  We call an $n$-dimensional lattice polytope $\Delta$ ($n\ge2$)
  \emph{exceptional}, if it is a simplex which is the $(n-2)$-fold
  pyramid over the $2$-dimensional basic simplex multiplied by $2$.
\end{defi}
\begin{figure}[ht]
  \centering
  \includegraphics[width=5cm]{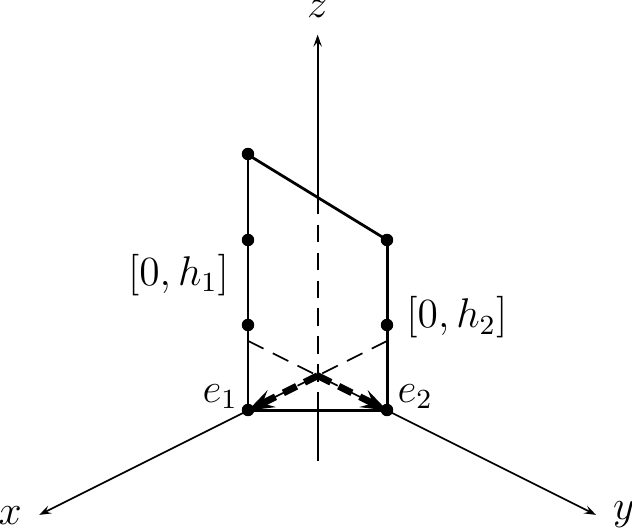}\hspace*{3cm}
  \includegraphics[width=3cm]{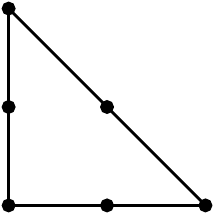}
  \caption{Illustration to the notion of Lawrence prism and to the
    notion of exceptional polytope}
  \label{fig:IllustrExceptionalPolytope}
\end{figure}

\begin{thm}[{\cite[Theorem 2.5]{BatyrevNill:LatticePolytopes}}]
  Let $\Delta$ be a $d$-dimensional lattice polytope. Then
  $\deg(h^*_\Delta)\le 1$ if and only if $\Delta$ is an exceptional
  simplex or a Lawrence prism.
\end{thm}

It remains to consider the case $k>1$, which we will assume from now
on. By \cite[Corollary
3.16]{BeckRobins:ComputingTheContinuousDiscretely}, $h_1^*=|\Delta\cap
M|-d-1$, so $\Delta$ must be a simplex. For $d$-dimensional lattice 
simplices the coefficients
$h_k^*$ of the $h^*$-polynomial $h_{\Delta}^*(t) = \sum_{k=0}^d h_k^* t^k$ 
have a simple combinatorial description: Let
$\mathbf{v}_0,\ldots,\mathbf{v}_d$ be the vertices of $\Delta$. Then
\begin{align*}
  \Pi_\Delta \coloneqq \left\{ \sum_{i=0}^d\lambda_i(\mathbf{v}_i,1) :
    0\le\lambda_i<1 \right\} \subset M_\R\oplus\R
\end{align*}
is called the \emph{fundamental parallelepiped} of $\Delta$.  By
\cite[Corollary 3.11]{BeckRobins:ComputingTheContinuousDiscretely}),
$h_k^*$ equals the number of lattice points in the fundamental
parallelepiped $\Pi_\Delta$ with last coordinate equal to $k$.

\begin{prop}
  Let $M$ be a lattice of rank $d$ and $\Delta\subseteq M_\R$ a
  $d$-dimensional lattice polytope with $h^*$-binomial of degree
  $k>1$.  Then $k\le(d+1)/2$.
\end{prop}
\begin{proof}
  Recall that $\Delta$ is a simplex. Let
  $\mathbf{v}_0,\ldots,\mathbf{v}_d$ be the vertices of
  $\Delta$. Assume that $k>(d+1)/2$. Then there exist
  $0\le\lambda_i<1$ for $i=0,\ldots,d$ such that
  $\sum_{i=0}^d\lambda_i(\mathbf{v}_i,1)\in M\oplus\Z$ has last
  coordinate strictly bigger than $(d+1)/2$. The last coordinate of
  \begin{align*}
    \sum_{i=0}^d\fractional{1-\lambda_i}(\mathbf{v}_i,1)\in M\oplus\Z
  \end{align*}
  is strictly smaller than $(d+1)/2$, so $h_\Delta^*$ is not a
  binomial. Contradiction.
\end{proof}
We have studied the case $k=(d+1)/2$ in \cite{BH:GeneralizedWhite}. In
this case, the dimension of the lattice simplex $d=2k-1$ is odd. We
give a summary of the classification result.
 
\begin{thm}\label{thm:GeneralizedWhite}
  Let $\Delta$ be a $(2k-1)$-dimensional lattice simplex which is not
  a lattice pyramid. The following two statements are equivalent:
  \begin{enumerate}
  \item The $h^*$-binomial of $\Delta$ has degree $k$.
  \item $\Delta$ is isomorphic to the Cayley polytope $\Delta_1*\ldots
    *\Delta_k$ of $1$-dimensional empty lattice simplices
    $\Delta_i\subset M'_\R$ where $M'$ is a $k$-dimensional lattice.
  \end{enumerate}
\end{thm}

\section{\texorpdfstring{The group $G_\Delta$ associated to a simplex $\Delta$}{The group G associated to a simplex D}}
\label{sec:G_Delta}

Let $M\cong\Z^d$ be a $d$-dimensional lattice. To a $d$-dimensional
lattice simplex $\Delta\subseteq M_\R$ with vertices
$\mathbf{v}_0,\ldots\mathbf{v}_d\in M$, we associate a group
\begin{align*}
  G_\Delta=\{\diag(e^{2\pi ix_0},\ldots,e^{2\pi
    ix_d}):x_i\in[0,1[,\sum_{i=0}^dx_i(\mathbf{v}_i,1)\in
  M\oplus\Z\}\subset\SL_{d+1}(\C).
\end{align*}

The composition of this group is component wise multiplication. Below,
we will mostly need an additive structure. That is why, we introduce
\begin{align*}
  \Lambda_\Delta=\{ \mathbf{x}=(x_0, \ldots, x_d) \in(\R/\Z)^{d+1}: \fractional{x_i}\in[0,1[,
  \sum_{i=0}^d\fractional{x_i}(\mathbf{v}_i,1) \in M\oplus\Z\}
\end{align*}
where $\left(\R/\Z\right)^{d+1}$ denotes the $(d+1)$-dimensional real
torus and where $\fractional{y}$ denotes the unique representative in
$[0,1[$ of $y\in\R/\Z$. To be more precise, the \emph{fractional part}
is the function $\fractional{\cdot}:\R\to[0,1[$ which associates to
$x\in\R$ the unique representative in $[0,1[$ of $x+\Z\in\R/\Z$. From
this we get a well-defined function
$\fractional{\cdot}:\R/\Z\to[0,1[$.

\begin{rem}
  It is clear that $G_\Delta$ and $\Lambda_\Delta$ are uniquely
  determined by each other. Indeed, they give the same group but
  $G_\Delta$ has a multiplicative structure while $\Lambda_\Delta$ has
  an additive structure.
\end{rem}

\begin{rem}
  There is a bijection between lattice points in the fundamental
  parallelepiped $\Pi_\Delta$ and elements in $\Lambda_\Delta$
  \begin{align*}
    \Lambda_\Delta\to\Pi_\Delta\cap(M\oplus\Z);\mathbf{x}\mapsto\sum_{i=0}^d\fractional{x_i}(\mathbf{v}_i,1).
  \end{align*}
  One might think of $\Lambda_\Delta$ as the lattice points in the
  fundamental parallelepiped. But instead of remembering the lattice
  points $\sum_{i=0}^d\fractional{x_i}(\mathbf{v}_i,1)$, we keep track
  of the coefficients $x_0,\ldots,x_d$ in the linear combination.
\end{rem}

Let $\mathscr{S}$ be the set of isomorphism classes of $d$-dimensional
lattice simplices with a chosen order of the vertices. Then
\begin{align*}
  \mathscr{Simpl}\to\{\text{finite
    subgroups}\;\Lambda\subset\left(\R/\Z\right)^{d+1}\};
  \Delta\mapsto\Lambda_\Delta
\end{align*}
is a bijection. We define the inverse map. Let
$\Lambda\subseteq\left(\R/\Z\right)^{d+1}$ be a finite group. Consider
the natural projection map
$\pi:\R^{d+1}\to\left(\R/\Z\right)^{d+1}$. The preimage
$M\coloneqq\pi^{-1}(\Lambda)\subseteq\R^{d+1}$ is a
$(d+1)$-dimensional lattice such that $\Z^{d+1}\subseteq M$ has finite
index. We denote the standard basis of $\Z^{d+1}$ by
$\mathbf{e}_1,\ldots,\mathbf{e}_{d+1}$. Then
$\Delta_\Lambda\coloneqq\conv(\mathbf{e}_1,\ldots,\mathbf{e}_{d+1})\subset
M_\R$ is a $d$-dimensional lattice simplex with respect to the affine
lattice $\mathrm{aff}(\mathbf{e}_1,\ldots,\mathbf{e}_{d+1})\cap M$.

The following theorem is easy to prove.
\begin{thm}
  The maps $\Delta\mapsto\Lambda_\Delta$ and
  $\Lambda\mapsto\Delta_\Lambda$ are inverse to each other. In
  particular, a lattice simplex is uniquely determined up to
  isomorphism by its group $\Lambda_\Delta$.
\end{thm}

\begin{exa}
  Let $G \subseteq\SL_3(\C)$ be the subgroup of order $4$ consisting
  of diagonal matrices with entries $1$ and $-1$. The singular locus
  of the quotient $\C^3/G$ consists of $3$ irreducible curves having
  one common point which comes from $1$-dimensional fixed point sets
  $(\C^3)^g$ for $3$ non-unit elements $g \in G$. Using methods of
  toric geometry one can identify all minimal resolutions of the
  non-isolated Gorenstein quotient singularities of $\C^3/G$ with all
  triangulations of the lattice triangle
  \begin{align*}
    \Delta = \{ (x_1,x_2) \in \R^2: x_1, x_2 \ge 0, x_1 + x_2 \leq 2\}
  \end{align*}
  whose $h^*$-polynomial has the form $1 + 3t$ and $G_\Delta=G$.
\end{exa}

We collect some basic observations concerning the group
$\Lambda_\Delta$ of a lattice simplex $\Delta$. Recall from the
introduction that the $h^*$-polynomial stays unchanged under pyramid
constructions. In contrast, the group $\Lambda_\Delta$ behaves as
follows.

\begin{prop}
  Let $\Delta$ be a $d$-dimensional lattice simplex. The following
  statements are equivalent:
  \begin{enumerate}
  \item $\Delta$ is a pyramid;
  \item there exists $i=0,\ldots,d$ such $x_i=0+\Z$ for all
    $\mathbf{x}\in \Lambda_\Delta$.
  \end{enumerate}
\end{prop}
\begin{proof}
  The implication $(1)\Rightarrow(2)$ is obvious.
 
  For the other direction, let $x_d=0+\Z$ for all
  $\mathbf{x}\in\Lambda_\Delta$. It suffices to show that
  $\Delta_{\Lambda_\Delta}$ is a pyramid. Let
  $\pi:\R^{d+1}\to\left(\R/\Z\right)^{d+1}$ be the natural projection
  map. We have $M\coloneqq\pi^{-1}(\Lambda_\Delta)=M'\oplus\Z$ where
  $M'=\mathrm{pr}(M)$ for $\mathrm{pr}:\R^{d+1}\to\R^d$ the projection
  onto the first $d$ coordinates. From this it follows that
  $\Delta_{\Lambda_\Delta}$ is a pyramid with apex $\mathbf{e}_{d+1}$.
\end{proof}

By \cite[Corollary 3.11]{BeckRobins:ComputingTheContinuousDiscretely},
we have
\begin{prop}\label{prop:coeff_h_poly}
  Let $\Delta$ be a $d$-dimensional lattice simplex. Then the
  coefficients of the $h^*$-polynomial
  $h^*_\Delta=\sum_{k=0}^dh_k^*t^k$ can be computed as follows
  \begin{align*}
    h_k^*=|\{(x_0,\ldots,x_d)\in\Lambda_\Delta:\sum_{i=0}^d\fractional{x_i}=k\}|.
  \end{align*}
\end{prop}

\section{The case \texorpdfstring{$1<k<(d+1)/2$}{1<k<(d+1)/2}}
\label{sec:k_in_between}

Consider a $d$-dimension lattice simplex $\Delta$ with $h^*$-binomial
of degree $1<k<(d+1)/2$ which is not a pyramid over a
lower-dimensional simplex. Let $\mathbf{v}_0,\ldots,\mathbf{v}_d$ be
the vertices of $\Delta$. Let
$\mathbf{0}\neq\mathbf{x}=(x_0,\ldots,x_d)\in \Lambda_\Delta$.  We can
uniquely express $\fractional{x_i}=a_i/n$ for a positive integer $n$
and non-negative integers $a_i<n$ subject to the condition
$\gcd(a_0,\ldots,a_d,n)=1$.  By rearranging the coordinates, we may
assume that
\begin{align*}
  a_0\cdots a_l\neq0\qquad\text{and}\qquad a_{l+1}=\ldots=a_d=0
\end{align*}
for $l=0,\ldots,d$.

\begin{prop}\label{prop:number_of_zeros}
  $l=2k$; in other words: the number of nonzero coordinates $a_i/n$ is
  a constant.
\end{prop}
\begin{proof}
  We have $\mathbf{0}\neq
  -\mathbf{x}=((n-a_0)/n+\Z,\ldots,(n-a_d)/n+\Z)\in\Lambda_\Delta$. By
  Proposition \ref{prop:coeff_h_poly}, we obtain
  \begin{align*}
    2k=\sum_{i=0}^d\fractional{\frac{n-a_i}{n}}+\sum_{i=0}^d\frac{a_i}{n}=\sum_{i=0}^l\frac{n-a_i}{n}+\sum_{i=0}^l\frac{a_i}{n}=l.
  \end{align*}
\end{proof}

\begin{cor}\label{cor:a_i_coprime_to_n}
  $\gcd(a_i,n)=1$ for $i=0,\ldots,d$.
\end{cor}
\begin{proof}
  Assume by contradiction that $\gcd(a_0,n)=q$. Then
  \begin{align*}
    \mathbf{0}\neq\frac{n}{q}\cdot\mathbf{x}=\left(\frac{a_0}{q}+\Z,\ldots,\frac{a_d}{q}+\Z\right)\in\Lambda_\Delta.
  \end{align*}
  But $a_0/q+\Z=0+\Z$ and there are at most $l-1$ nonzero coefficients
  contradicting the previous proposition.
\end{proof}

Let $\mathbf{0}\neq\mathbf{y}=(y_0,\ldots,y_d)\in
\Lambda_\Delta$. Like above, we can write $y_i=b_i/m$ for a positive
integer $m$ and non-negative integers $b_i<m$ coprime to $m$.

\begin{prop}\label{prop:Denominator_coeff_are_the_same}
  $m=n$.
\end{prop}

To be able to easily formulate the proof, we introduce the following
notion.
\begin{defi}
  Let $\mathbf{z}=(z_0,\ldots,z_d)\in \Lambda_\Delta$. The
  \emph{support} of $\mathbf{z}$ is given by
  \[
  \supp(\mathbf{z})\coloneqq\{i=0,\ldots,d:y_i\neq0+\Z\}.
  \]
\end{defi}

\begin{proof}[Proof of Proposition
  \ref{prop:Denominator_coeff_are_the_same}]
  We distinguish two cases.

  Assume that $\supp(\mathbf{x})\neq\supp(\mathbf{y})$. Consider
  \begin{align*}
    \mathbf{0}\neq\mathbf{x}+\mathbf{y}=\left(\left(\frac{a_0}{n}+\frac{b_0}{m}\right)+\Z,\ldots,\left(\frac{a_d}{n}+\frac{b_d}{m}\right)+\Z\right)\in\Lambda_\Delta.
  \end{align*}
  By Proposition \ref{prop:number_of_zeros}, there exists
  $i=0,\ldots,d$ such that $(a_i/n+b_i/m)+\Z=0+\Z$, \ie
  \begin{align*}
    \frac{a_i}{n}+\frac{b_i}{m}=1.
  \end{align*}
  The assertion follows by the fact that $(a_i,n)=1=(b_i,m)$.

  Assume that $\supp(\mathbf{x})=\supp(\mathbf{y})$. Since for all
  $\mathbf{0}\neq\mathbf{z}\in \Lambda_\Delta$,
  $|\supp(\mathbf{z})|=2k<d+1$, there exists $i=0,\ldots,d$ such that
  $z_i=0$. Since $\Delta$ is not a pyramid there exists
  $\mathbf{0}\neq\mathbf{z}\in \Lambda_\Delta$ with
  $\supp(\mathbf{x})\neq\supp(\mathbf{z})$. We may write $z_i=c_i/l$
  for a positive integer $l$ and non-negative integers $c_i<l$ coprime
  to $l$. By the same argument as above applied to $\mathbf{x}$ and
  $\mathbf{z}$, it follows that $l=n$. Applying this argument once
  again to $\mathbf{z}$ and $\mathbf{y}$ yields $l=m$. Hence $n=l=m$.
\end{proof}

\begin{cor}\label{cor:n_is_prime}
  $n=p$ is a prime number.
\end{cor}
\begin{proof}
  We may assume $n\neq2$. Assume by contradiction that $n$ is not
  prime, \ie there is a nontrivial divisor $d\mid n$.  Since $\Delta$
  is not a pyramid there exists $\mathbf{0}\neq\mathbf{z}\in
  \Lambda_\Delta$ with $\supp(\mathbf{x})\neq\supp(\mathbf{z})$. We
  can write $z_i=c_i/n$ for non-negative integers $c_i<n$ coprime to
  $n$.  By Proposition \ref{prop:number_of_zeros}, there exists
  $i\in\supp(\mathbf{x})\cap\supp(\mathbf{z})$ such that
  $a_i/n+c_i/n=1$.  Let $a_i'$ be an integer such that
  $a_i'a_i=1\in\Z/n\Z$. Consider
  $\mathbf{w}\coloneqq(d-c_i)a_i'\cdot\mathbf{x}+\mathbf{z}\in\Lambda_\Delta$. The
  $i$th coordinate of $\mathbf{w}$ is equal to $d/n+\Z$. Contradiction
  to Corollary \ref{cor:a_i_coprime_to_n}
\end{proof}

Our purpose is to relate the classification of $d$-dimensional lattice
polytopes $\Delta$ with an $h^*$-binomial of degree $1<k<(d+1)/2$ to
coding theory using linear codes $L_\Delta \subseteq \F_p^{d+1}$.
\begin{defi}\label{def:Hamming_weight_AND_constant_weight}
  A linear subspace $L\subseteq\mathds{F}_p^d$ of dimension $k$ will
  be called a \emph{linear code} of \emph{dimension} $k$ and
  \emph{block length} $d$.  Let $\mathbf{x}$ be a vector in
  $\mathds{F}_p^{d+1}$. The \emph{(Hamming) weight} of $\mathbf{x}$ is
  the number of nonzero coordinates in $\mathbf{x}$, \ie
  $\omega(\mathbf{x})=|\{i=1,\ldots,d+1:x_i\neq0\}|$. A linear subspace
  $L\subseteq\mathds{F}_p^{d+1}$ has \emph{constant weight} if every
  nonzero vector has the same weight.
\end{defi}
\begin{defi}
  Let $L\subseteq\F_p^d$ be a linear code. For every $\mathbf{x}\in
  L$, let $x_i$ be the unique integer representative between $0$ and
  $p-1$ of the $i$th coordinate of $\mathbf{x}$. Then the \emph{age}
  of $\mathbf{x}$ is given by
  \begin{align*}
    \age(\mathbf{x})=\sum_{i=1}^dx_i\in\N.
  \end{align*}
  Let $l\in\N$. We say that $L$ has \emph{constant age $l$}, if for
  all $\mathbf{0}\neq\mathbf{x}\in L$: $\age(\mathbf{x})=l$.
\end{defi}
\begin{rem}
  The name \enquote{age} is inspired by a definition of Ito and Reid
  (see \cite[Theorem 1.3]{Reid:McKay}).
\end{rem}

We have the following connection between age and weight:
\begin{prop}\label{prop:const_age_implies_const_weight}
  If a linear code $L\subseteq\F_p^d$ has constant age, then it has
  constant weight and for all $\mathbf{0}\neq\mathbf{x}\in L$:
  $2\age(\mathbf{x})=p\omega(\mathbf{x})$.
\end{prop}
\begin{proof}
  Let $\mathbf{0}\neq\mathbf{x}\in L$. Then
  \begin{align*}
    \age(\mathbf{x})=\age(-\mathbf{x})=\sum_{\substack{i=1\\x_i\neq0}}^d(p-x_i)=p\omega(\mathbf{x})-\age(\mathbf{x}).
 \end{align*}
\end{proof}
\begin{rem}
  Constant weight does not imply constant age. Indeed, a linear
  subspace $L\subseteq\F_p^d$ of dimension $1$ has constant weight but
  not constant age in general.
\end{rem}

\begin{thm}\label{thm:Props_of_linear_code_L_Delta}
  Let $\Delta\subseteq\R^d$ be a $d$-dimensional lattice simplex with
  $h^*$-binomial $h_\Delta^*=1+h_k^*t^k$ of degree $1<k<(d+1)/2$.
  Assume that $\Delta$ is not a pyramid over a lower-dimensional
  simplex. Then there exists a prime number $p$ such that
  $\Lambda_\Delta$ can be identified with a linear code
  $L_\Delta\subseteq\F_p^{d+1}$ of constant age $kp$. In particular,
  $L_\Delta$ has constant weight $2k$ and
  $\vol(\Delta)=1+h_k^*=|L_\Delta|=p^r$ where $r=\dim L_\Delta$.
\end{thm}
\begin{proof}
  By the statements \ref{prop:number_of_zeros} --
  \ref{cor:n_is_prime}, it follows that there is a prime number $p$
  such that $\Lambda_\Delta$ can be identified with a linear code
  $L_\Delta\subseteq\F_p^{d+1}$. By Proposition
  \ref{prop:coeff_h_poly}, it follows that $L_\Delta$ has constant age
  $kp$ and so, by Proposition
  \ref{prop:const_age_implies_const_weight}, $L_\Delta$ has constant
  weight $2k$.
 
  By \cite[Corollary
  3.21]{BeckRobins:ComputingTheContinuousDiscretely}),
  $\vol(\Delta)=1+h_k^*$. By Proposition \ref{prop:coeff_h_poly},
  $|L_\Delta|=1+h_k^*$. Since $L_\Delta$ is a linear subspace over
  $\F_p$, it follows that $|L_\Delta|=p^r$ where $r=\dim L_\Delta$.
\end{proof}

We will use Theorem \ref{thm:Props_of_linear_code_L_Delta} to
characterize $d$-dimensional lattice simplices $\Delta$ with
$h^*$-binomial of degree $1<k<(d+1)/2$ by investigating their
corresponding linear codes $L_\Delta$ more closely.

\section{Elements of Coding Theory}
\label{sec:coding_theory}
In this section, we will recall the notions and statements from coding
theory used later on. Fix a prime number $p$.

A linear map $f:\mathds{F}_p^d\to\mathds{F}_p^d$ is a \emph{monomial
  transform} if there exists a permutation $\sigma\in S_d$ and
elements $\tau_1,\ldots,\tau_d\in\mathds{F}_p^*$ such that
\[
f(x_1,\ldots,x_d)=(\tau_1x_{\sigma(1)},\ldots,\tau_dx_{\sigma(d)}).
\]
In other words $f$ is the composition of a permutation of the
coordinates followed by a coordinate wise dilation. Two linear codes
$L_1,L_2\subseteq\mathds{F}_p^d$ are called \emph{equivalent} if there
exists a monomial transform $f:\mathds{F}_p^d\to\mathds{F}_p^d$ such
that $f(L_1)=L_2$. This gives a equivalence relation on the set of
linear codes in $\mathds{F}_p^d$.

\begin{defi}
  Let $L\subseteq\mathds{F}_p^d$ be a linear code of dimension $k$ and
  block length $d$. A matrix $A\in\mathds{F}_p^{k\times d}$ is
  called a \emph{generator matrix} of $L$ if the row space is the
  given code, \ie
  $L=\{\mathbf{x}A:\mathbf{x}\in\mathds{F}_p^k\}$. In other words:
  To get a generator matrix of $L$ just take a basis of $L$ and
  arrange these vectors as the rows of a matrix $A$.
\end{defi}

Next we introduce the class of linear codes which we are interested
in.

\begin{defi}
  Fix a natural number $r$. Let $m = (p^r-1)/(p-1)$ be the number of
  points in $(r-1)$-dimensional projective space over $\F_p$.
  Consider $A$ an $r \times m$-matrix over $\F_p$ whose columns
  consist of one nonzero vector from each $1$-dimensional subspace of
  $\F_p^r$. Then $A$ forms the generator matrix of the \emph{simplex
    code} $L(A)$ of dimension $r$ and block length $m$. It is evident
  from the definition that all the linear codes which are produced by
  this procedure are equivalent. Furthermore, the simplex codes are the
  dual codes to the well-known Hamming codes.
\end{defi}

\begin{defi}
  Let $L\subseteq\mathds{F}_p^d$ be a linear code of dimension $r$ and
  block length $d$. Let $A\in\mathds{F}_p^{r\times d}$ be a generator
  matrix for $L$. If there is a linear code
  $L'\subset\mathds{F}_p^{d'}$ of dimension $r$ and block length
  $d'<d$ with generator matrix $A'\in\mathds{F}_p^{r\times d'}$ such
  that $A$ is equal to the $d/d'$-replication of the matrix $A'$, \ie
  \[
  A=\left(A',A',\ldots,A'\right)
  \]
  then $L$ will be called the \emph{$d/d'$-replication} of the linear
  code $L'$.
\end{defi}

\begin{thm}[{\cite{MacWilliams:ErrorCorrectingCodes}}, {\cite{WardWood:EquivalenceOfCodes}}]\label{thm:MacWilliams}
  Two linear codes $L_1, L_2 \subset \F_p^d$ are equivalent if an only
  if there exists a linear isomorphism $f:L_1 \to L_2$ that preserves
  weights.
\end{thm}
\begin{proof}
  If $L_1$ and $L_2$ are equivalent, then there is a monomial
  transform $f:\F_p^d\to\F_p^d$ which maps $L_1$ onto $L_2$. Obviously
  $f$ preserves weights.
 
  Next assume that there is a linear isomorphism $f:L_1\to L_2$ which
  preserves weights. First, we need to choose a non-trivial character
  $\chi:\F_p \to \C^*$ of $(\F_p, +)$. This is done by choosing a
  complex $p$-th root of unity $\zeta$ and defining
  $\chi(\overline{a}) = \zeta^a$ for $a \in \{ 0, 1, \ldots, p-1\}$.

  Denote by $\mathfrak{X}(\F_p)$ the group of all characters of
  $(\F_p, +)$.  For any $a \in \F_p$ we consider the multiplication
  map $\mu_a: \F_p \to \F_p; x \mapsto ax$. Then

 \[
 \F_p \to \mathfrak{X}(\F_p) , a \mapsto \chi \circ \mu_a
 \]
 is an isomorphism of finite abelian groups. This is a special
 instance of the isomorphism
 \begin{align*}
   \Hom(\F_p^d,\F_p)\to \mathfrak{X}(\F_p^d);\lambda\mapsto
   \chi\circ\lambda
 \end{align*}
 where $\mathfrak{X}(\F_p^d)$ denotes the group of characters of the
 group $(\F_p^d,+)$.

 Let $\mathbf{x}\in L_1$ and set $\mathbf{y}\coloneqq
 f(\mathbf{x})$. We can express the weights of $\mathbf{x}$ and
 $\mathbf{y}$ using the character $\chi$:
 \[
 \omega( \mathbf{x} ) = \sum_{i=1}^d \left( 1 - \frac{1}{q} \sum_{ a\in
     \F_p} \chi(ax_i) \right)\quad\text{and}\quad \omega( \mathbf{y} ) =
 \sum_{i=1}^d \left( 1 - \frac{1}{q} \sum_{ b\in \F_p} \chi(by_i)
 \right).
 \]
 Since $f$ preserves weights, we get the equality
 \[
 \sum_{i = 1}^d \sum_{ a\in \F_p} \chi(ax_i) = \sum_{i = 1}^d \sum_{
   b\in \F_p} \chi(by_i)
 \]
 We denote the projection onto the $i$th coordinate by
 $p_i:L_1\to\F_p$. Then the previous equation can be written as an
 equation of characters:
 \[
 \sum_{i = 1}^d \sum_{ a\in \F_p} \chi(a\cdot p_i) = \sum_{i = 1}^d
 \sum_{ b\in \F_p} \chi(b\cdot p_i\circ f)
 \]
 By subtracting $d$ times the constant $1$ function on both sides
 yields:
 \[
 \sum_{i = 1}^d \sum_{ a\in \F_p^*} \chi(a\cdot p_i) = \sum_{i =
   1}^d \sum_{ b\in \F_p^*} \chi(b\cdot p_i\circ f)
 \]
 Let $i=1$ and $a=1$. Since the characters of $\F_p^d$ are linearly
 equivalent, there exists $\sigma(1)\in\{1,\ldots,d\}$ and
 $\tau_1\in\F_p^*$ such that $\chi\circ p_1 = \chi(\tau_1 \cdot
 p_{\sigma(1)}\circ f)$. By the isomorphism
 $\Hom(\F_p^d,\F_p)\to\mathfrak{X}(\F_p^d)$ from above, we get that
 $p_1=\tau_1\cdot p_{\sigma(1)}\circ f$. Hence
 \begin{align*}
   \sum_{a\in\F_p^*}\chi(a\cdot
   p_1)=\sum_{b\in\F_p^*}\chi((b\tau_1)\cdot p_{\sigma(1)}\circ
   f).
 \end{align*}
 By subtracting this equation from the equation above, we get
 \begin{align*}
   \sum_{i = 2}^d \sum_{ a\in \F_p^*} \chi(a\cdot p_i) =
   \sum_{\substack{i = 1\\i\neq\sigma(1)}}^d \sum_{ b\in \F_p^*}
   \chi(b\cdot p_i\circ f)
 \end{align*}
 Inductively, we obtain a permutation $\sigma\in S_d$ and nonzero
 scalars $\tau_1,\ldots,\tau_d\in\F_p^*$ such that
 $p_i=\tau_i\cdot p_{\sigma(i)}\circ f$. We define a monomial
 transform
 \begin{align*}
   F:\F_p^d \to \F_p^d; (x_1, \ldots, x_d) \mapsto
   (\tau_{\sigma^{-1}(1)}^{-1}x_{\sigma^{-1}(1)}, \ldots,
   \tau_{\sigma^{-1}(d)}^{-1}x_{\sigma^{-1}(d)}).
 \end{align*}
 Then, by the choice of $\tau_i$ and $\sigma$, we obtain
 $\left.F\right|_{L_1}=f$. In particular, $F$ induces an equivalence
 of codes.
\end{proof}

\begin{cor}
  Any two $r$-dimensional codes $L_1 , L_2 \subset \F_p^d$ of the same
  constant weight are equivalent.
\end{cor}
\begin{proof}
  Choose bases $\mathbf{v}_1,\ldots,\mathbf{v}_r$ of $L_1$ and
  $\mathbf{w}_1,\ldots,\mathbf{w}_r$ of $L_2$. Let $f:L_1\to L_2$ be
  the linear isomorphism which maps $\mathbf{v}_i$ onto
  $\mathbf{w}_i$. Since $L_1$ and $L_2$ are two constant weight codes
  of the same weight, the linear function preserves
  weights. Then by the previous Theorem, $L_1$ and $L_2$ are
  equivalent.
\end{proof}

\begin{prop}\label{prop:weight_of_simplex_code}
  An $r$-dimensional simplex code has constant weight
  $p^{r-1}$.
\end{prop}
\begin{proof}
  Consider the first row of a generating matrix $A$ for a simplex
  code.  Then it contains exactly $p^{r-1}$ non-zero entries, because
  there exist exactly $p^{r-1}$ points in the projective space
  $\P^{r-1}(\F_p)$ having non-zero first homogeneous coordinate.
\end{proof}

\begin{thm}[{\cite{Bonisoli:EquidistantLinearCodes}},
  {\cite{WardWood:EquivalenceOfCodes}}]
  \label{thm:Bonisoli}
  Every $r$-dimensional code $L \subset \F_p^d$ of constant weight
  such that no coordinate is $0$ for all vectors in $L$ is equivalent to an
  $m$-fold replication of $r$-dimensional simplex codes.
\end{thm}
\begin{proof}
  Let $A$ be the generating $(r \times n)$-matrix for the code $L
  \subset\F_p^d$.  Without loss of generality we assume that $A$ has
  no zero-columns. Let $B \in GL(r,\F_p)$ be an arbitrary invertible
  $(r \times r)$-matrix. Then $BA$ is another generating matrix for
  $L$ and $B$ defines a natural bijective linear map $L \to L$. Since
  $L$ has constant weight, by Theorem \ref{thm:MacWilliams},
  the columns of $BA$ are obtained by a monomial transformation of
  columns of $A$. On the other hand, we can choose $B$ in such a way
  that the first column of $BA$ will be any prescribed nonzero vector
  in $\F_p^d$. Therefore, every nonzero vector in $\F_p^d$ is
  proportional to some column in $A$, i.e., $A$ contains a $(r\times
  (p^r-1)/(p-1))$-submatrix which determines a $r$-dimensional simplex
  code. Since any simplex code has constant weight, the
  remaining $(d-(p^r-1)/(p-1))$ columns of $A$ also generate an
  $r$-dimensional linear code $L' \subset \F_p^{d-(p^r-1)/(p-1)}$ of
  constant weight.  We can apply to $L'$ the same
  arguments. By this procedure, we decompose $L$ into a sequence of
  $m$ $r$-dimensional simplex codes.
\end{proof}

\begin{cor}
  Let $\Delta\subseteq\R^d$ be a $d$-dimensional lattice simplex with
  $h^*$-binomial of degree $1<k<(d+1)/2$. Assume that $\Delta=p^r$ for a prime number $p$ and a positive integer $r$. The
  number of vertices of $\Delta$ is a multiple of
  $(p^r-1)/(p-1)$.
\end{cor}
\begin{proof}
  By Theorem \ref{thm:Props_of_linear_code_L_Delta}, the linear code $L_\Delta$ is an $r$-dimensional linear
  subspace of $\F_p^{d+1}$. By Theorem \ref{thm:Bonisoli}, $d+1$ is an
  integer multiple of $(p^r-1)/(p-1)$. The number of vertices of
  $\Delta$ equals to $d+1$.
\end{proof}

\begin{cor}\label{cor:arithmetic_equation}
  With the assumptions of the previous corollary we have
  \begin{align*}
  2k(p^r-1)=(d+1)(p-1)p^{r-1}.
  \end{align*}
\end{cor}
\begin{proof}
  By Theorem \ref{thm:Bonisoli}, the linear code $L_\Delta$ is
  equivalent to a replicated simplex code
  $H\subseteq\F_p^{(p^r-1)/(p-1)}$, \ie $d+1=m(p^r-1)/(p-1)$ for a
  positive integer $m$. In particular, we have
  \begin{equation}
    m=\frac{(d+1)(p-1)}{p^r-1}.
  \end{equation}
  By Proposition \ref{prop:weight_of_simplex_code}, the weight of the simplex code $H$ is equal to $p^{r-1}$. Thus, the
  weight of $L_\Delta$ is $m p^{r-1}$. By Theorem
  \ref{thm:Props_of_linear_code_L_Delta}, the weight of
  $L_\Delta$ equals to $2k$.
\end{proof}

\section{Proof of the Main theorem}
\label{sec:proof_of_main_thm}

By Theorem \ref{thm:Props_of_linear_code_L_Delta}, a $d$-dimensional lattice simplex $\Delta$ with $h^*$-binomial of degree $1<k<(d+1)/2$ and a chosen order of its vertices is determined up to isomorphism by its associated linear code $L_\Delta\subseteq\F_p^{d+1}$ of constant age $kp$. So it is enough to classify linear codes in $\F_p^{d+1}$ of constant age.

If $p=2$, then, by Proposition \ref{prop:const_age_implies_const_weight}, the weight-function and the age-function coincide. Hence, our classification result follows by the Theorem of Bonisoli (see Theorem \ref{thm:Bonisoli}).
\begin{thm}
  The $d$-dimensional lattice simplices $\Delta$ with
  $h^*$-binomial of degree $1<k<(d+1)/2$ and $\mathrm{vol}(\Delta)=2^r$ for a positive integer $r$ which are not a pyramid over a lower-dimensional simplex are in correspondence with linear constant weight codes in $\F_p^{d+1}$ of weight $2k$. These codes are classified by the Theorem of Bonisoli. Furthermore, the numbers
  $k, d, r$ are related by the equation:
  \begin{align*}
    2^{r-2}(d+1)=k(2^r-1)
  \end{align*}
\end{thm}

For $p>2$, the age-function and the weight-function differ in general. We consider a construction to get linear codes of constant age.
\begin{prop}\label{prop:constr_codes_const_age}
 Fix a prime number $p>2$ and a positive integer $r$. Set $m=(p^r-1)/(p-1)$. Let $A_1,\ldots,A_s$ be $(r\times m)$-generator matrices of $r$-dimensional simplex codes. Then the linear code with generator matrix:
 \begin{align*}
  A\coloneqq(A_1,-A_1,\ldots,A_s,-A_s)
 \end{align*}
 has constant age $sp^r$.
\end{prop}

\begin{thm}
  The $d$-dimensional lattice simplices $\Delta$ with $h^*$-binomial of degree $1<k<(d+1)/2$ and $\mathrm{vol}(\Delta)=p^r$ for a prime number $p$ and a positive
  integer $r$ which are not a pyramid over a
  lower-dimensional simplex are in correspondence with linear codes $L_\Delta\subseteq\F_p^{d+1}$ of constant age $kp$.  The numbers $p,d,k, r$ are related by
  the equation
  \begin{align*}
    (p^r-p^{r-1})(d+1)=2k(p^r-1)
  \end{align*}
 and the generator matrix of those codes can be written up to permutation of the columns in the
  form
  \begin{align*}
    (A_1, -A_1, A_2, -A_2, \ldots, A_s, -A_s ),
  \end{align*}
  where $s=k/p^{r-1}$ and $A_1, \ldots, A_s$ are generator matrices of
  $r$-dimensional simplex codes.
\end{thm}
\begin{proof}
The result is a direct consequence of Theorem \ref{thm:Props_of_linear_code_L_Delta}, Corollary \ref{cor:arithmetic_equation}, Proposition \ref{prop:constr_codes_const_age} and the following theorem.
\end{proof}

\begin{thm}
  Let $L \subseteq \F_p^n$ be an $r$-dimensional linear code of
  constant age such that no coordinate is $0$ for all vectors in
  $L$. Then $n = m(p^r-1)/(p-1)$ for some even number $m = 2s$ and up
  to a permutation of the columns the generating $(r \times n)$-matrix
  of $L$ can be written in the form
  \begin{align*}
    (A_1, -A_1, A_2, -A_2, \ldots, A_s, -A_s ),
  \end{align*}
  where $A_1, \ldots, A_s$ are generating matrices for $r$-dimensional
  simplex codes.
\end{thm}

\begin{proof}
  Let $A$ be a generating $(r \times n)$-matrix of the $r$-dimensional
  linear code $L \subseteq \F_p^n$. By Proposition
  \ref{prop:const_age_implies_const_weight}, $L$ has constant
  weight. Since $A$ does not contain zero-columns, by Theorem
  \ref{thm:Bonisoli}, it is a sequence of $m$ matrices of
  $r$-dimensional simplex codes.  In particular, there exist exactly
  $m$ columns in $A$ representing any given point in $\P^r(\F_p)$, we
  have
  \begin{align*}
    n = m \frac{p^r-1}{p-1}.
  \end{align*}
  and $L$ has constant weight $mp^{r-1}$ (see Proposition
  \ref{prop:weight_of_simplex_code}).
  
  So, in order to prove the theorem, it is sufficient to show that the
  columns of the matrix $A$ can be reordered in pairs of the type
  $\{X, -X\}$.
  \begin{rem}\label{rem:wlg_p_odd}
    Observe that if $p=2$, then this is obvious since the only nonzero
    element in $\F_2$ is $1$ and $1+1=0$. In the following, we will
    assume that $p$ is an odd prime.
  \end{rem}

  The standard Terminal Lemma \cite{Reid:YoungPersonsGuide}, which was
  used in the generalized theorem of White \cite{BH:GeneralizedWhite},
  claims that up to some zero entries the last statement holds true
  for every row of $A$. However, different rows of $A$ may define
  different decompositions of coordinates in pairs.  The main
  difficulty is to guarantee that this decomposition into pairs holds
  true for all rows of $A$ simultaneously.

  We want to reformulate the last statement using the finite field
  $\F_q$ with $q = p^r$ elements. For this, we choose an $\F_p$-basis
  $\mathbf{v}_1, \ldots, \mathbf{v}_n$ of $\F_q$ and identify the
  $r$-dimensional columns of $A$ with some elements $\alpha_i$ of
  $\F_q^*$. So the generating matrix $A$ can be identified with a
  row $(\alpha_1, \ldots, \alpha_n) \in \left(\F_q^*\right)^n$.
  
  \begin{prop}\label{prop:reformulation_of_main_thm}
    Our goal is to show that $n=2s(p^r -1)/(p-1)$ is an even integer
    (for $s$ an integer) and up to reordering the coordinates, we can
    write the matrix $A$ in the form
    \begin{equation}
      (\alpha_1, -\alpha_1, \alpha_2, -\alpha_2, \ldots, \alpha_{n/2}, - \alpha_{n/2}).
      \label{version2}
    \end{equation}
  \end{prop}

  Denote by $T$ the matrix of the $\F_p$-bilinear form
  \begin{align*}
    \Tr: \F_q \times \F_q \to \F_p, \;\; (x,y) \mapsto \Tr(xy)
  \end{align*}
  in the basis $\mathbf{v}_1, \ldots, \mathbf{v}_n$. Since $\F_q$ is
  separable over $\F_p$, the matrix $T$ is non-degenerate and the
  $r\times n$-matrix $TA$ is also a generating matrix of the
  $r$-dimensional linear code $L$.  Since $TA$ can be obtained from
  $A$ by row operations, the statement of equation (\ref{version2})
  for $A$ is equivalent to the same statement for $TA$.  Any nonzero
  element $\mathbf{v} \in L$ can be obtained as product of a nonzero
  $(1 \times r)$-matrix $\Lambda = (\lambda_1, \ldots, \lambda_r)$
  with the $r\times n$-matrix $TA$.  If we consider the matrix
  $\Lambda \in \F_p^r \setminus \{ 0\}$ as an element $\lambda$ of the
  multiplicative group $\F_q^*$, then any nonzero vector
  $\mathbf{v} \in L$ can be written as
  \[
  \mathbf{v} = ( \Tr(\lambda \alpha_1), \Tr(\lambda \alpha_2), \ldots,
  \Tr(\lambda \alpha_n) ) \in \F_p^n
  \]
  for some $\lambda \in \F_q^*$.

  Since $L$ has constant age, we get $2\age(\mathbf{v})=\omega(\mathbf{v})$
  for all $\mathbf{0}\neq\mathbf{v}\in L$ (see Proposition
  \ref{prop:const_age_implies_const_weight}). In particular, we have
  \[
  \sum_{i =1}^n B_1 \left( \frac{\Tr(y \alpha_i)}{p} \right) = 0, \;\;
  \forall y \in \F_q^*
  \]
  where $B_1$ denotes the \emph{$1$st (periodic) Bernoulli function}
  which maps a real number $x$ to
  \begin{align*}
    B_1(x)=
    \begin{cases}
      \fractional{x}-\frac{1}{2} & x\not\in\Z\\
      0 & x\in\Z
    \end{cases}
  \end{align*}
  where $\fractional{x}$ denotes the \emph{fractional part} of $x$,
  \ie $\fractional{x}=x-\floor{x}$ where $\floor{x}$ is the biggest
  integer which is smaller than or equal to $x$.

  Let $G:= \F_q^*$ be the cyclic group of even order $p^r -1$
  ($p$ is an odd prime by Remark \ref{rem:wlg_p_odd}). We consider
  $V:= \C[G]$ the regular representation space of $G$ over $\C$
  together with the canonical basis $\sigma_g$ $(g\in G)$. The
  $\C$-space $V$ splits into a direct sum of $1$-dimensional
  $G$-invariant subspaces $V_i$ $ (i =1, \ldots, |G|)$ corresponding
  to $p^r-1$ different complex characters of $G$.  The $1$-dimensional
  $G$-invariant subspace $V_i$ corresponding to a character $\chi_i$
  is generated by the vector
  \[
  e_i \coloneqq \frac{1}{|G|} \sum_{g \in G} \chi(g^{-1})
  \sigma_g\in\C[G].
  \]

  To any character $\chi:G \to \C^*$ one associates a complex
  number
  \[
  B_{1, \chi}^{(r)} \coloneqq \sum_{g \in G} \chi(g) B_1 \left( \frac{
      \Tr(g)}{p} \right).
  \]
  \begin{rem}
    Observe that $B_{1,\chi}^{(1)}$ coincides with the classical
    generalized Bernoulli number (see
    \cite[p. 31]{Washington:CyclotomicFields}).
  \end{rem}

  A character $\chi: G \to \C^*$ is called {\it odd} if $\chi(-1)
  = -1$. There exist exactly $(p^r -1)/2$ odd characters of $G$.  We
  need the following statement, which will be proved in section \ref{sec:non_vanishing_of_B_1_chi}.

  \begin{thm}\label{thm:Nonvanishing_of_B1chi}
    Let $\chi: \F_q^* \to \C^*$ be a an odd character.  Then
    $B_{1, \chi}^{(r)} \neq 0$.
  \end{thm}

  Let $\{\sigma_g^* \}_{g \in G}$ be the dual basis of the dual space
  $V^*$. We consider in $V$ the subspace $U$ generated by the elements
  $\{ S(x) \}_{x \in G}$
  \[
  S(x) \coloneqq \sum_{ g \in G} B_1 \left( \frac{\Tr(xg)}{p} \right)
  \sigma_g
  \]
  Denote by $U^\perp$ the orthogonal complement in $V^*$ with respect
  to the canonical pairing
  \[
  \langle * , * \rangle \; : \; V \times V^* \to \C.
  \]
  Then $(p^r-1)/2$ linearly independent elements belong to $U^\perp$,
  because
  \[
  \langle S(x), e^*_g + e^*_{-g} \rangle = B_1 \left(
    \frac{\Tr(xg)}{p} \right) + B_1 \left( \frac{\Tr(-xg)}{p} \right)
  = 0.
  \]
  In order to show that $\{e^*_g + e^*_{-g}:g\in G\}$ forms a basis of
  $U^\perp$, we need to show that $\dim U^\perp = (p^r -1) - \dim U
  \leq (p^r -1)/2$, or, equivalently, that $\dim U \geq (p^r
  -1)/2$. We construct explicitly $(p^r-1)/2$ linearly independent
  vectors $u_i$ corresponding to odd characters $\chi_i$ of $G$:
  \begin{align*}
    u_i\coloneqq |G| B_{1, \chi_i}^{(r)} e_i &= B_{1, \chi_i}^{(r)}
    \sum_{g \in G} \chi_i(g^{-1}) \sigma_g = \sum_{ g \in G}
    \chi_i(g^{-1}) B_{1, \chi_i}^{(r)}
    \sigma_g\\
    &= \sum_{g \in G} \sum_{h \in G} \chi_i(g^{-1}h) B_1 \left(
      \frac{\Tr(h)}{p} \right) \sigma_g
  \end{align*}
  and show that they are contained in $U$. Using the substitution $g'
  = g^{-1}h$, this follows from
  \[
  u_i = \sum_{g\in G} \sum_{g'\in G} \chi_i(g') B_1 \left(
    \frac{\Tr(gg')}{p} \right) \sigma_g = \sum_{g' \in G} \chi_i(g')
  S(g') \in U.
  \]
  Now we can prove Proposition
  \ref{prop:reformulation_of_main_thm}. The sum $\sum_{i=1}^n
  e^*_{\alpha_i}$ belongs to $U^\perp$, because
  \[
  \langle S(x), \sum_{i=1}^n e^*_{\alpha_i} \rangle = \sum_{i=1}^n B_1
  \left( \frac{\Tr(x\alpha_i)}{p} \right) = 0\quad\text{for all}\;x\in
  G.
  \]
  This means that $e^*_g$ and $e^*_{-g}$ appear in the sum
  $\sum_{i=1}^n e^*_{\alpha_i}$ with the same multiplicity and we can
  divide the sequence $\alpha_1. \ldots, \alpha_n$ into pairs of
  elements of $G = \F_q^*$ with opposite sign.
\end{proof}

\section{\texorpdfstring{Nonvanishing of $B_{1,\chi}^{(r)}$}{Nonvanishing of B(1,chi)(r)}}
\label{sec:non_vanishing_of_B_1_chi}

In this section we will prove Theorem \ref{thm:Nonvanishing_of_B1chi}. Fix a prime power $q=p^r$ and an odd character $\chi:\F_q^*\to\C^*$.
Let us consider the square of the absolute value of $B_{1,\chi}^{(r)}$:
\begin{align*}
  |B_{1,\chi}^{(r)}|^2 =B_{1,\chi}^{(r)} \cdot \overline{B_{1,\chi}^{(r)}}
  =\sum_{a\in\F_q^*} \sum_{b\in\F_q^*} \chi(ab^{-1}) B_1\left( \frac{\Tr(a)}{p} \right) B_1\left( \frac{\Tr(b)}{p} \right)
\end{align*}
Using the substitution $c=ab^{-1}$, we get
\begin{align*}
  |B_{1,\chi}^{(r)}|^2&=\sum_{c\in\F_q^*} \chi(c) \sum_{a\in\F_q^*} B_1\left( \frac{\Tr(a)}{p} \right) B_1\left( \frac{\Tr(ac^{-1})}{p} \right). 
\end{align*}

We can extend the inner sum over $a$ to contain $0\in\F_q$, 
since $\Tr(0)=0$. Thus, we have to understand the sum
$\sum_{a\in\F_q} B_1( \Tr(a)/p )B_1( \Tr(ac^{-1})/p )$ for all $c\in\F_q^*$.

\begin{prop}\label{prop:B1chi_arithmetic_equation}
 We distinguish two cases:
  \begin{enumerate}
  \item If $c\in\F_p\subseteq\F_q$, then
    \[
    \sum_{a\in\F_q}B_1\left(\frac{\Tr(a)}{p}\right)B_1\left(\frac{\Tr(ac)}{p}\right)=p^{r-1}\sum_{a=1}^{p-1}B_1
    \left(\frac{a}{p}\right) B_1\left(\frac{ac}{p}\right).
    \]
  \item If $c\in\F_q\setminus\F_p$, then
    \[
    \sum_{a\in\F_q}B_1\left(\frac{\Tr(a)}{p}\right)B_1\left(\frac{\Tr(ac)}{p}\right)=0.
    \]
  \end{enumerate}
\end{prop}
Let us grant the previous result for a moment. We get

\begin{align*}
  |B_{1,\chi}^{(r)}|^2 = p^{r-1}\sum_{c\in\F_p^*}\chi(c)\sum_{a=1}^{p-1}B_1\left(\frac{a}{p}\right)
  B_1\left(\frac{ac^{-1}}{p}\right)
  \end{align*}
 
 Using the substitution $b=ac^{-1}$, we obtain
 \begin{align*}
  |B_{1,\chi}^{(r)}|^2 &= p^{r-1}\sum_{b\in\F_p^*}\chi(ab^{-1})\sum_{a=1}^{p-1}B_1\left(\frac{a}{p}\right)
  B_1\left(\frac{b}{p}\right)\\
  &=p^{r-1}\left(\sum_{a=1}^{p-1}\chi(a)B_1\left(\frac{a}{p}\right)\right)\cdot\left(\sum_{b=1}^{p-1}\overline{\chi}(b)B_1\left(\frac{b}{p}\right)\right)=p^{r-1}\cdot|
  B_{1,\left.\chi\right|_{\F_p}}^{(1)}|^2.
\end{align*}
By \cite[Chapter 4, p. 38]{Washington:CyclotomicFields}), it follows that $B_{1,\left.\chi\right|_{\F_p}}^{(1)}\neq0$. This proves Theorem \ref{thm:Nonvanishing_of_B1chi}.

\begin{proof}[Proof of Proposition \ref{prop:B1chi_arithmetic_equation}]
For the first part,
observe that the trace map $\Tr:\F_q\to\F_p$ is a linear functional with kernel $W:=\ker(\Tr)$ having codimension $1$ in $\F_q$. We can find a vector $v\in\F_q$ such
that $\F_q=W\oplus\F_pv$. By using the fact that $\Tr(\lambda
a)=\lambda\Tr(a)$ for all $\lambda\in\F_p$ and all $a\in\F_q$, we
get
\begin{align*}
  \sum_{a\in\F_q}B_1\left(\frac{\Tr(a)}{p}\right)B_1\left(\frac{\Tr(ac)}{p}\right)&=\sum_{\lambda\in\F_p}\sum_{w\in
    W}B_1\left(\frac{\Tr(w+\lambda
      v)}{p}\right)B_1\left(\frac{c\Tr(w+\lambda v)}{p}\right)\\
  &=\sum_{\lambda\in\F_p}\sum_{w\in W}B_1\left(\frac{\lambda\Tr(v)}{p}\right)B_1\left(\frac{\lambda c\Tr(v))}{p}\right)
\end{align*}
  The product of the Bernoulli functions does not depend on $w\in W$, $|W|=p^{r-1}$ and $\Tr(v)\neq0$. Thus, by
    substituting $\lambda':=\lambda\Tr(v)$, we obtain
\begin{align*}
  \sum_{a\in\F_q}B_1\left(\frac{\Tr(a)}{p}\right)B_1\left(\frac{\Tr(ac)}{p}\right)=p^{r-1}\sum_{\lambda'=1}^{p-1}B_1\left(\frac{\lambda'}{p}\right)B_1\left(\frac{\lambda'
      c}{p}\right).
\end{align*}

For the second part, assume that
$c\in\F_q\setminus\F_p$. We get two linear functionals $\Tr:\F_q\to\F_p;a\mapsto\Tr(a)$ and $f:\F_q\to\F_p;a\mapsto\Tr(ac)$ with kernels $W_1 \coloneqq \ker(\Tr)$ and $W_2 \coloneqq \ker(f)$. We
claim that $W_1\neq W_2$. Assume by contradiction that $W_1=W_2$. Then, 
we can find $v\in\F_q$ such that $W_i\oplus\F_pv=\F_q$ for $i=1,2$. Using those two decompositions of $\F_q$, we obtain for all
$a\in\F_q$
\[
\frac{f(v)}{\Tr(v)}\Tr(a)=f(a)
\]
where the quotient $c' \coloneqq f(v)/\Tr(v)$ is an element of
$\F_p^*$. This implies that $f(a)=\Tr(c'a)$ for all
$a\in\F_q$. Since $\F_q$ is separable over $\F_p$, the $\F_p$-bilinear form 
\begin{align*}
\Tr:\F_q\times\F_q\to\F_p;(a,b)\mapsto\Tr(ab) 
\end{align*}
is not degenerate. Hence, the equality
\begin{align*}
 f(a)=\Tr(ac=\Tr(ac')\quad\text{for all}\;a\in\F_q
\end{align*}
yields $c'=c$. Contradiction.

Thus $W_1\cap W_2$ is a subspace of codimension $2$ in
$F_q$.
Thus we can find vectors $w_1\in
W_1$ and $w_2\in W_2$ such that 
\[
\F_q=(W_1\cap W_2)\oplus\F_p w_1\oplus\F_p w_2.
\]
Using this decomposition, we can write
\begin{align*}
  &\sum_{a\in\F_q}B_1\left(\frac{\Tr(a)}{p}\right)B_1\left(\frac{\Tr(ac)}{p}\right)\\
  &=\sum_{w\in W_1\cap
    W_2}\sum_{\lambda_1\in\F_p}\sum_{\lambda_2\in\F_p}B_1\left(\frac{\Tr(w+\lambda_1 w_1+\lambda_2 w_2)}{p}\right)B_1\left(\frac{\Tr(w+\lambda_1 w_1+\lambda_2 w_2)c}{p}\right)\\
  &=\sum_{w\in W_1\cap
    W_2}\sum_{\lambda_1\in\F_p}\sum_{\lambda_2\in\F_p}B_1\left(\frac{\lambda_2\Tr(w_2)}{p}\right)B_1\left(\frac{\lambda_1 c\Tr(w_1)}{p}\right)
 \end{align*}
    The product of the Bernoulli functions does
    not depend on $w\in W_1\cap W_2$ and $|W_1\cap W_2|=p^{r-2}$. Using the substitutions $\lambda_1':=\lambda_1 c\Tr(w_1)$
    and $\lambda_2':=\lambda_2\Tr(w_2)$, we get
 \begin{align*}
  \sum_{a\in\F_q}B_1\left(\frac{\Tr(a)}{p}\right)B_1\left(\frac{\Tr(ac)}{p}\right)&=p^{r-2}\sum_{\lambda'_1\in\F_p}\sum_{\lambda'_2\in\F_p}
  B_1\left(\frac{\lambda'_2}{p}\right)
  B_1\left(\frac{\lambda'_1}{p}\right)\\
  &=p^{r-2}\left(\sum_{\lambda'_1\in\F_p}
    B_1\left(\frac{\lambda'_1}{p}\right)\right)
  \left(\sum_{\lambda'_2\in\F_p}B_1\left(\frac{\lambda'_2}{p}\right)\right)
\end{align*}

It is straightforward to verify
\begin{align*}
 \sum_{\lambda\in\F_p}B_1\left(\frac{\lambda}{p}\right)=\sum_{\lambda=1}^{p-1}\left(\frac{\lambda}{p}-\frac{1}{2}\right)=\frac{p(p-1)}{2p}-\frac{p-1}{2}=0
\end{align*}

%
\end{proof}

\bibliographystyle{amsalpha}
\bibliography{ref}

\end{document}